\documentclass[a4paper,11pt]{amsart}
\usepackage{latexsym}
\usepackage{eepic,epic, amsmath}
\usepackage{graphics,epsfig}
\usepackage{graphicx}
\usepackage{verbatim}

\newcommand{\A}{\mathcal{A}}
\theoremstyle{plain}
\newtheorem{theorem}{Theorem}
\newtheorem*{theorem*}{Theorem}
\newtheorem{corollary}[theorem]{Corollary}
\newtheorem*{corollary*}{Corollary}
\newtheorem{lemma}[theorem]{Lemma}
\newtheorem*{lemma*}{Lemma}

\newtheorem*{proposition*}{Proposition}
\newtheorem{conjecture}{Conjecture}
\newtheorem*{conjecture*}{Conjecture}
\theoremstyle{definition}

\newtheorem*{definition*}{Definition}

\newtheorem*{example*}{Example}

\newtheorem*{problem*}{Problem}

\theoremstyle{remark}
\newtheorem{remark}{Remark}
\newtheorem*{remark*}{Remark}

\title{Crucial words for abelian powers}

\author[A. Glen]{Amy Glen}

\email[A. Glen]{amy.glen@gmail.com}

\author[B.V. Halld\'orsson]{Bjarni V. Halld\'orsson}

\email[B.V. Halld\'orsson]{bjarnivh@ru.is}

\author[S. Kitaev]{Sergey Kitaev}

\email[S. Kitaev]{sergey@ru.is}

\address{The Mathematics Institute, Reykjav\'ik University, Kringlan 1, 103 Reykjav\'ik, Iceland}



\date{December 25, 2008}

\begin{document}

\begin{abstract}
A word is {\em crucial} with respect to a given set of {\em prohibited words} (or simply  {\em prohibitions}) if it avoids the prohibitions but it cannot be extended to the right
by any letter of its alphabet without creating a prohibition.
A {\em minimal crucial word} is a crucial word of the
shortest length. A word $W$ contains an {\em abelian $k$-th power}
if $W$ has a factor of the form $X_1X_2\ldots X_k$ where $X_i$ is a
permutation of $X_1$ for $2\leq i\leq k$. When $k=2$ or $3$, one deals with
{\em abelian squares} and {\em abelian cubes}, respectively.

Evdokimov and Kitaev~\cite{EK} have shown that a minimal crucial
word over an $n$-letter alphabet $\A_n = \{1,2,\ldots, n\}$ avoiding abelian squares has length $4n-7$ for $n\geq 3$. 
In this paper, we
show that a minimal crucial word over $\A_n$ avoiding abelian cubes has length
$9n-13$ for $n\geq 5$, and it has length 2, 5, 11, and 20 for
$n=1,2,3$, and 4, respectively. Moreover, for $n\geq 4$ and $k\geq 2$, we give a
construction of length $k^2(n-1)-k-1$ of a crucial word over $\A_n$ avoiding abelian $k$-th powers. This construction gives the minimal length for $k=2$ and $k=3$. 
%We conjecture that this construction
%gives the minimal length for any $k\geq 4$ and sufficiently large
%$n$. 
For $k \geq 4$ and $n\geq 5$, we provide a lower bound for the length of crucial words over $\A_n$ avoiding abelian $k$-th powers. 
\end{abstract}

  \maketitle
  \thispagestyle{empty}

\section{Introduction}

Let $\A_n=\{1,2,\ldots,n\}$ be an $n$-letter alphabet and let $k \geq 2$ be an integer. 
A word $W$ over $\A_n$ contains a {\em $k$-th power} if $W$ has 
a factor of the form $X^k = XX\ldots X$ ($k$ times) for some non-empty
word $X$. A $k$-th power is {\em trivial} if $X$ is a single letter. For example, the word $V = 13243232323243$ contains the (non-trivial) $4$-th power $(32)^4 = 32323232$.   
A word $W$ contains an {\em abelian $k$-th power} if $W$ has a
factor of the form $X_1X_2\ldots X_k$ where $X_i$ is a permutation of $X_1$ for
$2\leq i\leq k$.  The cases $k=2$ and $k=3$ give us ({\em abelian}) {\em squares} and
{\em cubes}, respectively. For instance, the preceding word $V$ contains the abelian square $43232\thinspace 32324$ and the word $123\thinspace 312 \thinspace 213$ is an abelian cube. A word is ({\em abelian}) {\em $k$-power-free} if it {\em avoids} (abelian) $k$-th powers.  For example, the word $1234324$ is abelian cube-free, but not abelian square-free since it contains the abelian square $234\thinspace 324$. 

A word $W$ over $\A_n$ is {\em crucial} with respect to a given set of {\em prohibited words} (or simply {\em prohibitions}) if $W$ avoids the prohibitions, but $Wx$ does not avoid the prohibitions for
any $x\in\A_n$. A {\em minimal crucial word} is a crucial word of the
shortest length. For example, the word $W = 2 1 2 1 1$ (of length $5$) is crucial with respect to abelian cubes since it is abelian cube-free and the words $W1$ and $W2$ end with the abelian cubes $111$ and $21\thinspace 21 \thinspace 12$, respectively. Actually, $W$ is a minimal crucial word over $\{1,2\}$ with respect to abelian cubes. Indeed, one can easily verify that there does not exist any crucial abelian cube-free words over $\{1,2\}$ of length less than $5$.

Abelian squares were first introduced by Erd\H{o}s~\cite{pE61some}, who asked whether or not there exist words of arbitrary length over a fixed finite alphabet that avoid factors of the form $XX'$ where $X'$ is a permutation of $X$. This question has since been solved in the affirmative; see for instance \cite{aC98onth, aE68stro, vK92abel, pP70nonr} for work in this direction. Problems of this type were also considered by Zimin~\cite{Zimin}, who used the following sequence of words as a key tool.

The {\em Zimin word}~\cite{Zimin} $Z_n=X_n$ over $\A_n$ is defined recursively as
follows: $X_1=1$ and $X_n=X_{n-1}nX_{n-1}$ for $n \geq 2$. The first four Zimin words are:
\begin{align*}
X_1 &= 1, \\
X_2 &= 121, \\
X_3 &= 1213121, \\
X_4 &= 121312141213121.
\end{align*}
 The {\em
$k$-generalised Zimin word} $Z_n^k=X_n$ is defined as
$$X_1= 1^{k-1} = 11\ldots 1,\ X_n= (X_{n-1}n)^{k-1} X_{n-1} = X_{n-1}nX_{n-1}n\ldots nX_{n-1}$$
where the number of 1's, as well as the number of $n$'s, is $k-1$.
Thus $Z_n=Z_n^2$. It is easy to see that $Z_n^k$ avoids (abelian)
$k$-th powers and it has length $k^n-1$. Moreover, it is known that
$Z_n^k$ gives the length of a minimal crucial word avoiding $k$-th powers.

However, in the case of abelian powers the situation is not as well
studied. Crucial abelian square-free words (also called {\em left
maximal abelian square-free words}) of exponential length are given
in~\cite{C} and~\cite{EK}, and it is shown in~\cite{EK} that a
minimal crucial abelian square-free word over an $n$-letter alphabet
has length $4n-7$ for $n\geq 3$.

In this paper, we extend the study of crucial abelian
$k$-power-free words to the case of $k>2$.  In particular, we provide a
complete solution to the problem of determining the length of a minimal crucial abelian cube-free word (the case $k = 3$) and we conjecture a solution in the general case. More precisely, we show that a
minimal crucial word over $\A_n$ avoiding abelian cubes has length $9n-13$
for $n\geq 5$ (Corollary~\ref{min-length}), and it has length 2, 5, 11, and 20 for $n=1,2,3$, and
$4$, respectively. For $n \geq 4$ and $k\geq 2$, we give a construction of
length $k^2(n-1)-k-1$ of a crucial word over $\A_n$ avoiding abelian
$k$-th powers (see Theorem~\ref{opt-upper-k}). This construction gives the minimal length for $k = 2$ and $k= 3$, and we conjecture that this is also true for any $k\geq 4$ and sufficiently large~$n$. 
We also provide a rough lower bound for the length of minimal crucial words over $\A_n$ avoiding abelian $k$-th powers,  for $n\geq 5$ and $k\geq 4$ (see Theorem~\ref{rough-l-bound}). 

We let $\ell_k(n)$ denote the length of a minimal crucial abelian
$k$-power-free word over $\A_n$ and we denote by $|W|$ the length of a word
$W$.  For a crucial word $X$ over $\A_n$, we let $X=X_i\Delta_i$, 
where $\Delta_i$ is the minimal factor such that $\Delta_ii$ is a
prohibition. Note that we can rename letters, if needed, so we can assume that for any minimal crucial word $X$, one has
$$\Delta_1\subset\Delta_2\subset\cdots\subset\Delta_n=X$$
where ``$\subset$'' means (proper) {\em right factor} (or {\em suffix}). In other
words, $\Delta_i=Y_i\Delta_{i-1}$ for each $i=2,3,\ldots,n$ and
$Y_i$ is not empty. In what follows we will use $X_i$
and $Y_i$ as stated above.  We note that the definitions imply:
$$X=X_i\Delta_i=X_iY_i\Delta_{i-1}=X_{n-1}Y_{n-1}Y_{n-2}\ldots Y_2\Delta_1,$$
for any $i=2,3,\ldots,n-1$.  Furthermore, assuming we consider $k$-th powers, we
write $\Delta_ii=\Omega_{i,1}\Omega_{i,2}\ldots\Omega_{i,k}$, where the $k$ {\em blocks} $\Omega_{i,j}$ are equal up to permutation, and we denote by $\Omega'_{i,k}$ the block $\Omega_{i,k}$ without the rightmost $i$.

\section{Crucial words for abelian cubes}\label{cube}

\subsection{An upper bound for $\ell_3(n)$}\label{sub-u}

The fact that the 3-generalised Zimin word $Z_n^3$ is crucial with
respect to abelian cubes already gives us an upper bound of $3^n-1$ for
$\ell_3(n)$. In Theorem~\ref{cubes} below we improve this upper
bound to $3\cdot 2^{n-1} - 1$. We then discuss a further
improvement of the bound using a greedy algorithm, which gives
asymptotically $O((\sqrt{3})^n)$. This greedy algorithm
provides minimal crucial abelian cube-free words over $\A_n$ for $n=3$ and
$n=4$, while the construction in Theorem~\ref{cubes} is optimal only
for $n=3$. We also provide a construction of a crucial word of
length $9n-10$ which exceeds our lower bound by only 3
letters, for $n\geq 5$. Finally, we end this section with a construction of a crucial abelian cube-free word over $\A_n$ of length $9n-13$, which coincides with the lower bound  given in Theorem~\ref{lower} of Section~\ref{sub-l} for $n\geq 5$. 
%We keep the
%non-optimal constructions for $n\geq 5$ for the sake of collecting together
%different approaches to tackle such problems.

\begin{theorem}\label{cubes} One has that $\ell_3(n)\leq 3\cdot2^{n-1}-1$.\end{theorem}

\begin{proof} We construct a crucial abelian cube-free word $X=X_n$ iteratively as follows.
Set $X_1=11$ and assume $X_{n-1}$ has been constructed. Then do the following:
\begin{enumerate}\item Increase all letters of $X_{n-1}$ by 1 to obtain $X^+_{n-1}$.
\item Insert 1 after (to the right of) each letter of $X^+_{n-1}$ and adjoin one extra 1 to the right of the resulting
word to get $X_n$.\end{enumerate} For example, $X_2=21211$,
$X_3=31213121211$, etc. It is easy to verify that $|X_n|=3\cdot2^{n-1}-1$. We show by induction that $X_n$ avoids abelian cubes, whereas $X_nx$ does not avoid abelian cubes for any $x \in \A_n$.
Both claims are trivially true for $n=1$.  Now take $n \geq 2$. If $X_n$ contains an abelian cube,
then removing 1's from it, we would deduce that $X_{n-1}$ must also contain
an abelian cube, contradicting the fact that $X_1$ contains no abelian
cubes.

It remains to show that extending $X_n$ to the right by
any letter $x$ from $\A_n$ creates an abelian cube. If $x=1$ then we
get $111$ from the construction of $X_n$.
On the other hand, if $x>1$  then we swap the rightmost 1 with
the rightmost $x$ in $Xx$, thus obtaining a word where every other
letter is 1; removing all 1's and decreasing each of the remaining
letters by 1, we have $X_{n-1}(x-1)$, which contains an abelian
cube (by the induction hypothesis). 
\end{proof}

As a further improvement of Theorem~\ref{cubes} we sketch here,
without providing all the details, the work of a {\em greedy
algorithm}. Here, ``greedy'' means that we assume (by induction) that an ``optimal''
crucial abelian cube-free word $X_{n-1}$ over $\A_{n-1}$ has been constructed. Next (for the greedy step) we add just two $n$'s (the minimum possible), and then we add as few of the other letters as possible to build a crucial abelian cube-free word $X_n$ over $\A_n$. More precisely, we set $\Delta_1=11$ and assuming that
$\Delta_{i-1}$ is built, we consider the minimum set $T_i$ of
letters we must add (forming $Y_i$ and possibly updating
$\Delta_{i-1}$ by permuting letters in $Y_{i-1}$) to build
$\Delta_i$, for $i=2,3,\ldots,n$, which can then be turned into a
construction of a crucial abelian cube-free word $X_n$ over $\A_n$. It is easy to see
from the definitions of an abelian cube and a crucial word that 
$T_2$ must contain at least two 2's and at least one 1. Then $T_3$
must contain at least two 3's, at least one 2, and at least three
1's (the last statement follows from the fact that the two 3's in
$\Delta_3$ are supposed to be accompanied by at least one 1 in
$\Omega_{3,1}$ and $\Omega_{3,2}$ in
$\Delta_33=\Omega_{3,1}\Omega_{3,2}\Omega_{3,3}$, but without extra
1's, whose number of occurrences must be divisible by 3, we cannot
manage it). Running this type of argument, one can come up with
the following list of minimal requirements for $T_i$ for initial
values of $i$:
$$T_2=\{1,2,2\},T_3=\{1,1,1,2,3,3\},T_4=\{1,1,1,2,2,2,3,4,4\},$$
$$T_5=\{1,1,1,1,1,1,1,1,1,2,2,2,3,3,3,4,5,5\},$$
$$T_6=\{1,1,1,1,1,1,1,1,1,2,2,2,2,2,2,2,2,2,3,3,3,4,4,4,5,6,6\}, \mbox{etc.}$$
In particular, we observe that $(|T_{n}|)_{n\geq 1} = 2, 3, 6, 9, 18, 27, \ldots$ where
\[
|T_{2i+1}| = 2\cdot 3^i \quad \textrm{and} \quad |T_{2i+2}| = 3^i \quad \textrm{for all $i \geq 0$.}
\]
Hence, these considerations lead to a crucial abelian cube-free word over an $n$-letter alphabet of
length 
$$\sum_{j=0}^{\lfloor \frac{n-1}{2}\rfloor} 2\cdot 3^j + \sum_{j=1}^{\lfloor \frac{n}{2} \rfloor} 3^j.$$
Initial values for the lengths are 2, 5, 11, 20, 38, 65, $\ldots$~.
Furthermore, by observing that $\sum_{j=0}^{m} 3^j = \frac{3^{m+1} - 1}{2}$ for any integer $m \geq 0$ (which can be easily proved by induction), we deduce that the length of a crucial abelian cube-free word obtained in this way is asymptotically $O((\sqrt{3})^n)$. Even though this length is exponential, we {\em do} in fact obtain optimal values for $n\leq
4$. Below we list optimal abelian cube-free crucial words over $\A_n$ for $n=1,2,3,4$ of lengths
2,5,11, and 20, respectively:
$$11;$$
$$21211;$$
$$1 1 2 3 1 3 2 1 2 1 1;$$
$$4 2 1 3 1 2 1 4 2 3 1 2 1 1 3 2 1 2 1 1.$$

A construction giving the best possible upper bound for $n\geq 5$
can be easily described by examples, and we do this below (for
$n=4,5,6,7$; the construction does not work for $n\leq 3$).
We also provide a general description. The pattern in the
construction is easy to recognise. 
%We prove that the construction gives crucial abelian cube-free words for any $n\geq 5$ in Theorem~\ref{opt-upper} below.

\medskip
\noindent {\bf An almost optimal construction for crucial abelian
cube-free words.} The construction of the word $W_n$ for $n=4,5,6,7$
is shown below. We use spaces to separate the blocks $\Omega_{n,1}$,
$\Omega_{n,2}$, and $\Omega'_{n,3}$ in $W_n=\Delta_n$ for a more
pleasing visual representation.
\newpage
$$W_4 = 3 4 4 2 3 3 1 2 2\ 4 3 3 2 2 1 4 3 2\ 3 2 1 2 2 3 3 4$$
$$W_5=4 5 5 3 4 4 2 3 3 1 2 2\ 5 4 4 3 3 2 2 1 5 4 3 2\ 4 3 2 1 2 2 3
3 4 4 5$$
$$W_6=5 6 6 4 5 5 3 4 4 2 3 3 1 2 2\ 6 5 5 4 4 3 3 2 2 1 6 5 4 3 2\ 5
4 3 2 1 2 2 3 3 4 4 5 5 6$$
$$W_7=6 7 7 5 6 6 4 5 5 3 4 4 2 3 3 1 2 2\ 7 6 6 5 5 4 4 3 3 2 2 1 7 6
5 4 3 2\ 6 5 4 3 2 1 2 2 3 3 4 4 5 5 6 6 7$$

In general, the block $\Omega_{n,1}$ in
$W_n=\Delta_n=\Omega_{n,1}\Omega_{n,2}\Omega'_{n,3}$ is built by
adjoining the factors $i(i+1)(i+1)$ for
$i=n-1,n-2,\ldots,1$. The block $\Omega_{n,2}$ is built by adjoining together
the following factors: $n$, $xx$ for $n-1\geq x\geq 2$, and
$n(n-1)\ldots 2$. Finally, the block $\Omega'_{n,3}$ is built by adjoining the
factors $(n-1)(n-2)\ldots 1$, $xx$ for $2\leq x\leq n-1$, and finally $n$.

It is easy to see that $|W_n|=9n-10$. We omit the details of showing that
$W_n$ is crucial with respect to abelian cubes since this can be shown in a similar
manner for the construction below (see the proof of Theorem~\ref{opt-upper-k}).

\medskip
\noindent {\bf An optimal construction for crucial abelian cube-free
words.} The construction of the word $E_n$ for $n=4,5,6,7$ works as
follows. As above, we use spaces to separate the blocks $\Omega_{n,1}$,
$\Omega_{n,2}$, and $\Omega'_{n,3}$ in $E_n=\Delta_n$.
$$E_4 = 34423311 \ 34231134 \ 3233411$$
$$E_5 = 4 5 5 3 4 4 2 3 3 1 1\ 4 5 3 4 2 3 1 1 3 4 5 \ 4 3 2 3 3 4 4 5 1 1$$
$$E_6 = 5 6 6 4 5 5 3 4 4 2 3 3 1 1\ 5 6 4 5 3 4 2 3 1 1 3 4 5 6 \ 5 4 3 2 3 3 4 4 5 5 6 1 1$$
$$E_7 = 6 7 7 5 6 6 4 5 5 3 4 4 2 3 3 1 1 \  6 7 5 6 4 5 3 4 2 3 1 1 3 4 5 6 7 \  6 5 4 3 2 3 3 4 4 5 5 6 6 7 1 1$$

In general, the block $\Omega_{n,1}$ in
$E_n=\Delta_n=\Omega_{n,1}\Omega_{n,2}\Omega'_{n,3}$ is built by
adjoining the factors $i(i+1)(i+1)$ for
$i=n-1,n-2,\ldots,2$, followed by two 1's. The block $\Omega_{n,2}$
is built by adjoining the following factors: $i(i+1)$ for $i = n-1, n-2, \ldots, 2$, followed by $11$, and then  the factor $34\ldots (n-1)n$. Finally, the block $\Omega'_{n,3}$ is built by adjoining the factors
$(n-1)(n-2)\ldots 32$, then $xx$ for $3\leq x\leq n-1$, followed by
$n$, and finally two 1's.

By construction, we have
\[
E_{n} = \Omega_{n,1}\Omega_{n,2}\Omega'_{n,3}
\]
where $\Omega_{n,3} = \Omega'_{n,3}n$ and each $\Omega_{n,i}$  contains two $1$'s, one $2$, two $n$'s, and three $x$'s for $x = 3, \ldots, n-1$. Hence, it is easy to see that  $|E_n|=9n-13$. The fact that $E_n$ is crucial with respect to abelian cubes is proved in Theorem~\ref{opt-upper-k} where one needs to set $k=3$. % (see Remark~\ref{E_n}). 
Thus, a minimal crucial word avoiding abelian cubes has length at most $9n-13$ for $n\geq 4$. That is:

\begin{theorem}\label{opt-upper} For $n\geq 4$, we have $\ell_3(n)\leq 9n-13$.\end{theorem}
\begin{proof} See the proof of Theorem~\ref{opt-upper-k} where one needs to set $k = 3$ (in view of Remark~\ref{E_n}, later).
\end{proof}

\subsection{A lower bound for $\ell_3(n)$} \label{sub-l}

If $X=\Delta_n$ is a crucial word with respect to abelian cubes,
then clearly the number of occurrences of each letter except $n$ must be divisible by $3$, whereas the number of occurrences of $n$ is 2 modulo 3. We
sort in non-decreasing order the number of occurrences of the letters
$1,2,\ldots,n-1$ in $X$ to get a non-decreasing sequence of numbers $(a_1\leq a_2\leq
\ldots\leq a_{n-1})$. Notice that $a_i$ does not necessarily
correspond to the letter $i$. We denote by $a_0$ the number of
occurrences of the letter $n$. Also note that $a_0$ can be either larger or
smaller than $a_1$. By definitions, $|X|=\sum_{i=0}^{n-1}a_i$.

The word $E_n$ of length $9n-13$ in Subsection~\ref{sub-u} has the
following sequence of $a_i$'s: $(a_0,a_1,\ldots,
a_{n-1})=(5,3,6,9,\ldots, 9)$. In this subsection, we prove
that this sequence cannot be improved for $n\geq 5$, meaning that,
e.g., 5 cannot be replaced by 2, and/or 6 cannot be replaced by 3,
and/or 9('s) cannot be replaced by 3('s) or 6('s), no matter what construction we use to form a crucial word. This is a direct
corollary to the following four lemmas and is recorded in
Theorem~\ref{lower}. In the next four lemmas we use, without
explanation, the following facts that are easy to see from the
definitions. For any letter $x$ in a crucial abelian cube-free word
$X$:
\begin{itemize}
\item  the number of occurrences of $x$ in
$\Delta_x$ is 2 modulo 3 and the number of occurrences of any other letter, if any, in
$\Delta_x$ is divisible by 3;
\item if $x+1$ exists, then $Y_{x+1}$ in
$\Delta_{x+1}=Y_{x+1}\Delta_x$ contains 2 modulo 3 copies of $x+1$
and 1 modulo 3 occurrences of $x$, whereas the number of occurrences of any other letter, if any, in $Y_{x+1}$ is divisible by $3$.
\end{itemize}

Abusing notions, we think sometimes of words as sets, and use
``$\in$'' and ``$\subseteq$'' for ``occur(s)'' when the relative
order of letters is not important in the argument.

\begin{lemma}\label{lem1} For a crucial abelian cube-free word $X$, $|X|\geq 3$,
the sequence of $a_i$'s cannot contain 3,3. That is,
$(a_1,a_2)\neq(3,3)$.\end{lemma}

\begin{proof} 
Suppose that $x$ and $y$, with $x<y<n$, are two letters occurring in
$X$ exactly 3 times (each). We let $A_1=Y_nY_{n-1}\ldots Y_{y+1}$
and $A_2=Y_yY_{y-1}\ldots Y_{x+1}$ so that we have
$X=A_1A_2\Delta_x$. We must have the following distribution of $x$'s
and $y$'s: $y\in A_1$, $\{x,y,y\}\subseteq A_2$, and
$\{x,x\}\subseteq \Delta_x$. However, we get a contradiction, since
each of the blocks $\Omega_{n,2}$ and $\Omega'_{n,3}$  in
$X=\Delta_n=\Omega_{n,1}\Omega_{n,2}\Omega'_{n,3}$ must receive one
copy of $x$ and one copy of $y$, which is impossible (no $x$ can
exist between the two rightmost $y$'s).
\end{proof}

\begin{lemma} For a crucial abelian cube-free word $X$, $|X|\geq 4$, the sequence of $a_i$'s cannot contain 6,6,6.\end{lemma}

\begin{proof} We first prove the following fact. \\

\noindent {\bf A useful fact.} If $x$ and $y$, with $x<y<n$, are two
letters occurring in $X$ exactly 6 times (each), then $\Delta_x$
{\em cannot} contain 5 copies of $x$. Indeed, if this were the
case, then assuming $A_1=Y_nY_{n-1}\ldots Y_{y+1}$ and
$A_2=Y_yY_{y-1}\ldots Y_{x+1}$ giving $X=A_1A_2\Delta_x$, we have
that $A_1A_2$ has exactly one $x$ and at least three $y$'s,
contradicting the fact that each of the blocks $\Omega_{n,1}$, $\Omega_{n,2}$,
and $\Omega'_{n,3}$ must receive two $x$'s and two $y$'s. \\

Suppose that $x$, $y$, and $z$, with $x<y<z<n$, are three letters
occurring in $X$ exactly 6 times (each). We let
$A_1=Y_nY_{n-1}\ldots Y_{z+1}$, $A_2=Y_zY_{z-1}\ldots Y_{y+1}$, and
$A_3=Y_yY_{y-1}\ldots Y_{x+1}$ so that $X=A_1A_2A_3\Delta_x$. The
minimal requirements on the $A_i$ are as follows: $z\in A_1$,
$\{z,z,y\}\subseteq A_2$, $\{y,y,x\}\subseteq A_3$. Moreover, using
the useful fact above applied to $x$ and $y$, $\Delta_x$ contains
{\em exactly} two copies of $x$. The useful fact applied to $y$ and
$z$ guarantees that $A_1A_2$ contains 4 $y$'s (in particular,
$\Delta_x$ does not contain any $y$'s).

Looking at $X=\Delta_n=\Omega_{n,1}\Omega_{n,2}\Omega'_{n,3}$, we
see that each of the blocks $\Omega_{n,1}$, $\Omega_{n,2}$, and $\Omega'_{n,3}$
must receive 2 $x$'s, 2 $y$'s, and 2 $z$'s. Thus, in $A_3$, we must
have the following order of letters: $x,y,y$ and the boundary
between $\Omega_{n,2}$ and $\Omega'_{n,3}$ must be between $x$ and
$y$ in $A_3$. But then $\Delta_x$ entirely belongs to
$\Omega'_{n,3}$, so it cannot contain any $z$'s (if it would do so,
$\Delta_x$ would then contain 3 $z$'s which is impossible). On the
other hand, (exactly) 3 $z$'s must be in $A_3$ for $\Omega'_{n,3}$
to receive 2 $z$'s. Thus, $\Delta_y$ contains 2 $y$'s, 3 $z$'s and 3
$x$'s which is impossible by Lemma~\ref{lem1} applied to the word
$\Delta_y$ with two letters occurring exactly 3 times
(alt ernatively, one can see, due to the considerations above, that
no $z$ can be between the two rightmost $x$'s contradicting the
structure of $\Delta_y$).
\end{proof}

\begin{lemma} For a crucial abelian cube-free word $X$, $|X|\geq 4$, the sequence of $a_i$'s cannot contain 3,6,6.\end{lemma}

\begin{proof} Suppose that $x$ occurs exactly 3 times and $y$ and $z$ occur exactly 6 times (each) in $X$. We consider three
cases covering all the possibilities up to renaming $y$ and $z$.

\begin{description}
\item[Case 1] $z<y<x<n$. One can see that $\Delta_y$ does not contain $x$,
but it contains at least 3 $z$'s contradicting the fact that each of
$\Omega_{n,1}$, $\Omega_{n,2}$, and $\Omega'_{n,3}$ must receive 1
$x$ and 2 $z$'s. \\
\item[Case 2] $x<z<y<n$. We let
$A=Y_nY_{n-1}\ldots Y_{z+1}$ so that $X=A\Delta_z$. All three $x$'s
must be in $\Delta_z$, while $A$ must contain at least 3 $y$'s
contradicting the fact that each of the blocks $\Omega_{n,1}$, $\Omega_{n,2}$,
and $\Omega'_{n,3}$ must receive 1 $x$ and 2 $y$'s. \\
\item[Case 3] $z<x<y<n$. We let
$A_1=Y_nY_{n-1}\ldots Y_{y+1}$, $A_2=Y_yY_{y-1}\ldots Y_{x+1}$, and
$A_3=Y_xY_{x-1}\ldots Y_{z+1}$ so that $X=A_1A_2A_3\Delta_z$. The
minimal requirements on the $A_i$ and $\Delta_z$ are as follows: $y\in
A_1$, $\{x,y,y\}\subseteq A_2$, $\{z,x,x\}\subseteq A_3$, and
$\{z,z\}\subseteq\Delta_z$. The remaining 3 $y$'s cannot be in
$\Delta_z$ so as not to contradict the structure of $\Delta_x$ (it
would not be possible to distribute $x$'s and $y$'s in a proper
way). However, if the remaining 3 $y$'s are in $A_3$ then, not to
contradict the structure of $\Delta_x$ (no proper distribution of
$y$'s and $z$'s would exist), the remaining 3 $z$'s must be in
$\Delta_x$, which contradicts to the structure of
$X=\Delta_n=\Omega_{n,1}\Omega_{n,2}\Omega'_{n,3}$ (no proper
distribution of $y$'s and $z$'s would exist among the blocks $\Omega_{n,1}$,
$\Omega_{n,2}$, and $\Omega'_{n,3}$, each of which is supposed to
have exactly 2 copies of $y$ and 2 copies of $z$). Thus, there are
no $y$'s in $\Delta_x$, contradicting the structure of $\Delta_n$ (no
proper distribution of $y$'s and $x$'s would exist among the blocks
$\Omega_{n,1}$,$\Omega_{n,2}$, and $\Omega'_{n,3}$).
\end{description}\end{proof}

%\begin{lemma}\label{lem4-1} For a crucial abelian cube-free word $X$, $|X|\geq 5$, $$(a_0,a_1,a_2,a_3,a_4)\neq (2,3,9,9,9).$$\end{lemma}
%
%\begin{proof} We have that $n$ occurs twice in $X$. Assume that a letter $t$ occurs exactly 3 times. If $t\neq n-1$ then
%all three occurrences of $t$ are in $\Delta_{n-1}$ whereas the two
%occurrences of $n$ are in $Y_n$ (recall that
%$X=\Delta_n=Y_n\Delta_{n-1})$. We get a contradiction with the fact
%that $\Omega_{n,1}$ must contain 1 copy of $n$ and 1 copy of $t$.
%Thus, $n-1$ occurs exactly 3 times.
%
%Suppose that $x$, $y$, and $z$, $x<y<z<n-1$, are three letters
%occurring in $X$ exactly 9 times (each). Clearly, $\Delta_z$ does
%not contain $n-1$, and since $\Omega'_{n,3}$ must contain $n-1$,
%$\Delta_z$ entirely belongs to $\Omega'_{n,3}$. On the other hand,
%we know that $\Omega'_{n,3}$ must have three copies of each $x$,
%$y$, and $z$, and thus we are in the situation of Lemma~\ref{lem1}
%($\Delta_z$, viewed as $X$ in Lemma~\ref{lem1}, contains exactly 3
%copies of $x$ and three copies of $y$), which is impossible.
%\end{proof}

\begin{lemma}\label{lem4} For a crucial abelian cube-free word $X$, $|X|\geq 5$, $$(a_0,a_1,a_2,a_3,a_4)\neq (2,3,6,9,9).$$\end{lemma}

\begin{proof}
Suppose that $n$ occurs twice in $X$ and assume that a letter $t$ occurs exactly 3 times. If $t\neq n-1$, then
all three occurrences of $t$ are in $\Delta_{n-1}$ whereas the two
occurrences of $n$ are in $Y_n$ (recall that
$X=\Delta_n=Y_n\Delta_{n-1})$. This contradicts the fact
that $\Omega_{n,1}$ must contain 1 copy of $n$ and 1 copy of $t$.
Thus, the letter $n-1$ occurs exactly 3 times.

Now, assuming $x$, $y$, and $z$, with $x<y<z<n-1$, are three letters
occurring in $X$ $\{6,9,9\}$ times (we do not specify which letter
occurs how many times), we have, similar to the proof of
Lemma~\ref{lem4}, that $\Delta_z$ entirely belongs to
$\Omega'_{n,3}$. Moreover, the block $\Omega'_{n,3}$ has 2,3,3 occurrences of letters
$x,y,z$ (in some order). However, if $x$ or $y$ occur twice in
$\Omega'_{n,3}$, they occur twice in $\Delta_z$ leading to a
contradiction with $\Delta_z$'s structure. Thus $z$ must occur twice in $\Omega'_{n,3}$, and $x$ and $y$ occur 3 times (each) in
$\Omega'_{n,3}$. But then it is clear that $x$ and $y$ must occur 3
times (each) in $\Delta_z$, contradicting the fact
that $x$ and $z$ cannot be distributed properly in $\Delta_z$.
\end{proof}

\begin{theorem}\label{lower} For $n\geq 5$, we have $\ell_3(n)\geq 9n-13$. \end{theorem}

\begin{proof} This is a direct corollary to the preceding four lemmas, which tell us that any attempt to decrease
numbers in the sequence $(5,3,6,9,9,\ldots)$ corresponding to $E_n$
will lead to a prohibited configuration. \end{proof}

\begin{corollary} \label{min-length} For $n \geq 5$, we have $\ell_{3}(n) = 9n-13$.
\end{corollary}
\begin{proof} The result follows immediately from Theorems~\ref{opt-upper} and \ref{lower}.
\end{proof}

\begin{remark} Recall from Section~\ref{sub-u} that $\ell_3(n) = 2, 5, 11, 20$ for $n = 1,2,3,4$, respectively. For instance, the word $4 2 1 3 1 2 1 4 2 3 1 2 1 1 3 2 1 2 1 1$  is an optimal abelian cube-free crucial word of length $20$ ($= 2 + 3 + 6 + 9$). This can be proved using similar arguments as in the proofs of the Lemmas~\ref{lem1}--\ref{lem4}.
\end{remark}

\section{Crucial words for abelian $k$-th powers}\label{general}

\subsection{An upper bound for $\ell_k(n)$ and a conjecture}

The following theorem is a direct generalisation of
Theorem~\ref{cubes} and is a natural approach to obtaining an upper bound that
improves $k^n-1$ given by the $k$-generalised Zimin word $Z_n^k$.

\begin{theorem}\label{abcubes} For $k\geq 3$, we have $\ell_k(n)\leq
k\cdot(k-1)^{n-1}-1$.\end{theorem}

\begin{proof}
We proceed as in the proof of Theorem~\ref{cubes}, with the only
difference being that we put $k-2$ 1's to the right of each letter and one
extra 1 as the rightmost one. \end{proof}

We skip here analysis of the work of a greedy algorithm, and proceed
directly with the construction of a crucial abelian $k$-power-free word $W_{n,k}$ before describing
the construction of a similar word $D_{n,k}$ that we believe to be
optimal.

\medskip
\noindent {\bf A construction of a crucial abelian $k$-power-free word $W_{n,k}$, where $n\geq 4$ and $k\geq 3$.} We illustrate
each step of the algorithm by example, letting $k=n=4$.

The construction can be explained directly, but we introduce it
recursively, obtaining, for $n\geq 4$, $W_{n,k}$ from $W_{n,k-1}$
and using the abelian cube case $W_{n,3}=W_n$ as the basis. For
$n=4$,
$$W_{4,3}=\Omega_{4,1}\Omega_{4,2}\Omega'_{4,3}=3 4 4 2 3 3 1 2 2\  4 3 3 2 2 1 4 3 2\  3 2 1 2 2 3 3 4.$$
Assume that $W_{n,k-1}=\Omega_{n,1}\Omega_{n,2}\ldots
\Omega'_{n,k-1}$ is constructed and implement the following steps
to obtain $W_{n,k}$:

\begin{enumerate}

\item Duplicate $\Omega_{n,1}$ in $W_{n,k-1}$ to obtain the word $$W'_{n,k-1}=\Omega_{n,1}\Omega_{n,1}\Omega_{n,2}\ldots
\Omega'_{n,k-1}.$$ For $n=k=4$, $$W'_{4,3}=3 4 4 2 3 3 1 2 2\ 3 4
4 2 3 3 1 2 2\  4 3 3 2 2 1 4 3 2\  3 2 1 2 2 3 3 4$$

\item Append to the second $\Omega_{n,1}$ in $W'_{n,k}$ the factor $n(n-1)\ldots 2$ (in our example, 432) to obtain $\Omega_{n,2}$ in $W_{n,k}$,
and in each of the remaining blocks $\Omega_{n,i}$ in
$W'_{n,k-1}$, duplicate the rightmost occurrence of each letter $x$,
where $2\leq x\leq n$, to obtain $W_{n,k}$. For $n=k=4$,
$$W_{4,4}=3 4 4 4 2 3 3 3 1 2 2 2\ 3 4 4 2 3 3 1 2 2 4 3 2\
4 3 3 2 2 1 4 4 3 3 2 2\ 3 2 1 2 2 2 3 3 3 4 4.$$
\end{enumerate}

We provide two more examples here, namely $W_{5,4}$ and $W_{4,5}$,
respectively, so that the reader can check their understanding
of the construction:
\begin{small}
$$4 5 5 5 3 4 4 4 2 3 3 3 1 2 2 2\
4 5 5 3 4 4 2 3 3 1 2 2 5 4 3 2\ 5 4 4 3 3 2 2 1 5 5 4 4 3 3 2 2\ 4
3 2 1 2 2 2 3 3 3 4 4 4 5 5;$$
$$3 4 4 4 4 2 3 3 3 3 1 2 2 2 2\  3 4 4 4 2 3 3 3 1 2 2 2 4 3 2\  3 4 4 2 3 3 1 2 2 4 4 3 3 2 2\  4 3 3 2 2 1 4 4 4 3 3 3 2 2 2\
3 2 1 2 2 2 2 3 3 3 3 4 4 4.$$
\end{small}

It is easy to see that $|W_{n,k}|=k^2(n-1)-1$.
We omit the proof that $W_{n,k}$ is crucial with respect to abelian $k$-th powers since it is similar to the proof for the following word $D_{n,k}$, which is $k$ letters shorter than $W_{n,k}$ (see Theorem~\ref{opt-upper-k}).

\medskip
\noindent {\bf A construction of a crucial abelian $k$-power-free word $D_{n,k}$,
where $n\geq 4$ and $k\geq 2$.} As we shall see, the following construction of the word $D_{n,k}$ is optimal for $k =2, 3$. We believe that it is also optimal for any $k \geq 4$ and sufficiently large~$n$ (Conjecture~\ref{conj}).

As our basis for the construction of the word $D_{n,k}$, we use the following word $D_{n,2} = D_n$,  which is constructed as follows, for $n = 4,5,6,7$. (As previously, we use spaces to separate the blocks $\Omega_{n,1}$ and $\Omega'_{n,2}$ in $D_n=\Delta_n$.)
$$D_4 = 3 4 2 3 1  \ 3 2 3 1$$
$$D_5 = 4 5  3 4  2 3  1\ 4 3 2 3  4  1$$
$$D_6=5 6  4 5  3 4  2 3  1 \ 5 4 3 2 3  4  5 1$$
$$D_7=6  7 5 6  4 5  3 4 2 3 1\ 6 5 4 3 2 3 4 5  6 1$$

In general, the first block $\Omega_{n,1}$ in $D_n=\Delta_n=\Omega_{n,1}\Omega'_{n,2}$ is built by
adjoining the factors $i(i+1)$ for
$i=n-1,n-2,\ldots,2$, followed by the letter $1$. The second block $\Omega'_{n,2}$ is built by adjoining the factors $(n-1)(n-2)\ldots 432$, then $34\ldots (n-2)(n-1)$, and finally the letter $1$.

\begin{remark} The above construction coincides with the construction given in \cite[Theorem~5]{EK} for a minimal crucial abelian square-free word over $\A_n$ of length $4n-7$. In fact, the word $D_n$ can be obtained from the minimal crucial abelian cube-free word $E_n$ (defined in Section~\ref{sub-u}) by removing the second block in $E_n$ and deleting the rightmost copy of each letter except $2$ in the first and third blocks of $E_n$. 
\end{remark}

Now we illustrate each step of the construction for the word $D_{n,k}$ by example, letting $n= 4$ and $k=3$. The construction is very similar to that of $W_{n,k}$ and can be
explained directly, but we introduce it recursively, obtaining  $D_{n,k}$ from $D_{n,k-1}$ for $n\geq 4$. Our basis is $D_{n,2}=D_{n}$.
For $n=4$,
$$D_{4,2}=\Omega_{4,1}\Omega'_{4,2}= 3 4 2 3 1  \ 3 2 3 1.$$
Assume that $D_{n,k-1}=\Omega_{n,1}\Omega_{n,2}\ldots
\Omega'_{n,k-1}$ is constructed and implement the following steps
to obtain $D_{n,k}$:

\begin{enumerate}

\item Duplicate $\Omega_{n,1}$ in $D_{n,k-1}$ to obtain the word $$D'_{n,k-1}=\Omega_{n,1}\Omega_{n,1}\Omega_{n,2}\ldots
\Omega'_{n,k-1}.$$ For $n=4$ and $k=3$, $$D'_{4,3}= 3 4 2 3 1  \ 3 4 2 3 1  \ 3 2 3 1.$$

\item Append to the second $\Omega_{n,1}$ in $D'_{4,3}$ the factor
$134\ldots n$ (in our example, 134; in fact, any permutation of
$\{1,3,4,\ldots,n\}$ would work at this place) to obtain
$\Omega_{n,2}$ in $D_{n,k}$. In each of the remaining blocks 
$\Omega_{n,i}$ in $D'_{n,k-1}$, duplicate the rightmost occurrence
of each letter $x$, where $1\leq x\leq n-1$ {and} $x\neq 2$. Finally, in the last block of $D'_{n,k}$ insert the letter $n$ immediately before the leftmost $1$ to obtain the word $D_{n,k}$. For $n = 4$ and $k=3$, we have 
$$D_{4,3}=3 4 4 2 3 3 1 1 \ 3 4 2 3 1 1 3 4 \ 3 2 3 3 4 1 1.$$
\end{enumerate}

We provide five more examples here, namely $D_{5,3}$, $D_{5,4}$, $D_{4,4}$, $D_{4,5}$, and $D_{6,4}$, respectively, so that the reader can check their understanding
of the construction:
\begin{small}
$$4 5 5 3 4 4 2 3 3 1 1\ 4 5 3 4 2 3 1 1 3 4 5 \ 4 3 2 3 3 4 4 5 1 1;$$
$$4 5 5 5 3 4 4 4 2 3 3 3 1 1 1\  4 5 5 3 4 4 2 3 3 1 1 1 3 4 5 \ 4 5 3 4 2 3 1 1 1 3 3 4 4 5 5 \ 4 3 2 3 3 3 4 4 4 5 5 1 1 1;$$
$$3 4 4 4 2 3 3 3 1 1 1 \ 3 4 4 2 3 3 1 1 1 3 4 \ 3 4 2 3 1 1 1 3 3 4 4 \ 3 2 3 3 3 4 4 1 1 1;$$
$$3 4 4 4 4 2 3 3 3 3 1 1 1 1 \  3 4 4 4 2 3 3 3 1 1 1 1 3 4 \ 3 4 4 2 3 3 1 1 1 1 3 3 4 4 \ 3 4 2 3 1 1 1 1 3 3 3 4 4 4 \ 3 2 3 3 3 3 4 4 4 1 1 1 1;$$
$$5 6 6 6 4 5 5 5 3 4 4 4 2 3 3 3 1 1 1\  5 6 6 4 5 5 3 4 4 2 3 3 1 1 1 3 4 5 6 \ 5 6 4 5 3 4 2 3 1 1 1 3 3 4 4 5 5 6 6 \ 5 4 3 2 3 3 3 4 4 4 5 5 5 6 6 1 1 1.$$

\end{small}

\begin{remark} \label{E_n} By construction, $D_{n,3} = E_n$ for all $n\geq 4$.
\end{remark}

\begin{theorem}\label{opt-upper-k} For $n \geq 4$ and $k\geq 2$, we have $\ell_k(n)\leq k^2(n-1)-k-1$.\end{theorem}

\begin{proof} Fix $n \geq 4$ and $k \geq 2$. By construction, we have
\[
 D_{n,k}  =  \Omega_{n,1}\Omega_{n,2} \ldots \Omega_{n,k-1} \Omega'_{n,k}
 \]
 where $\Omega_{n,k} = \Omega'_{n,k}n$ and each $\Omega_{n,i}$ contains $(k-1)$ occurrences of the letter~$1$, one occurrence of the letter $2$, $(k-1)$ occurrences of the letter $n$, and $k$ occurrences of the letter $x$ for $x = 3, 4, \ldots, n-1$. Hence, it is easy to see that $|D_{n,k}|=k^2(n-1)-k-1$. We prove that $D_{n,k}$ is crucial with respect to abelian $k$-th powers; whence the result. The following facts, which are easily verified from the construction of $D_{n,k}$, will be useful in the proof.

\medskip
\noindent {\bf Fact 1:} In every block $\Omega_{n,i}$, the letter $3$ has occurrences before and after the single occurrence of the letter $2$.

\medskip
\noindent {\bf Fact 2:} In every block $\Omega_{n,i}$, all $(k-1)$ of the $1$'s occur after the single occurrence of the letter $2$ (as the factor $1^{k-1} = 11\ldots 1$).

\medskip
\noindent {\bf Fact 3:} For all $i$ with $2\leq i \leq k-1$, the block $\Omega_{n,i}$ ends with $n^{i-1} = nn\cdots n$ ($i-1$ times) and the other $(k-1-i+1)$ $n$'s occur (together as a string) before the single occurrence of the letter $2$ in $\Omega_{n,i}$. In particular, there are exactly $k-2$ occurrences of the letter $n$ between successive $2$'s in $D_{n,k}$.

\medskip
\noindent {\bf Freeness}: First we prove that $D_{n,k}$ is abelian $k$-power-free. Obviously, by construction, $D_{n,k}$ is not an abelian $k$-th power (as the number of occurrences of the letter $n$ is not a multiple of $k$) and $D_{n,k}$ does not contain any {\em trivial} $k$-th powers, i.e., $k$-th powers of the form $x^k = xx\ldots x$ ($k$ times) for some letter $x$. Moreover, each block $\Omega_{n,i}$ is abelian $k$-power-free. For if not, then according to the frequencies of the letters in the blocks, at least one of the $\Omega_{n,i}$ must contain an abelian $k$-th power consisting of exactly $k$ occurrences of the letter $x$ for all $x = 3, 4, \ldots, n-1$ and no occurrences of the letters $1$, $2$, and $n$. But, by construction, this is impossible because, for instance, the letter $3$ has occurrences before and after the letter $2$ in each of the blocks $\Omega_{n,i}$ in $D_{n,k}$ (by Fact~1).

Now suppose, by way of contradiction, that $D_{n,k}$ contains a non-trivial abelian $k$-th power, $P$ say. Then it follows from the preceding paragraph that $P$ overlaps at least two of the blocks $\Omega_{n,i}$ in $D_{n,k}$. We first show that $P$ cannot overlap three or more of the blocks  in $D_{n,k}$. For if so, then $P$ must contain at least one of the blocks, and hence $P$ must also contain all $k$ of the $2$'s. Furthermore, all of the $1$'s in each block occur after the letter~$2$ (by Fact~2), so there are $(k-1)^2 = k^2 -2k +1$ occurrences of the letter $1$ between the leftmost and rightmost $2$'s in $D_{n,k}$. Thus, $P$ must contain all $k(k-1) = k^2 - k$ of the $1$'s. Hence, since $\Omega'_{n,k}$ ends with $1^{k-1} = 11\ldots 1$ ($k-1$ times), we deduce that $P$ must end with the word
\[
  W = 23^{k-1}1^{k-1}\Omega_{n,2}\ldots \Omega_{n,k-1}\Omega'_{n,k},
\]
which contains $k$ of the $n$'s, $k(k-1) + (k-1) = k^2 - 1$ of the $3$'s, and $k(k-1)$ occurrences of the letter $x$ for $x = 4, \ldots, n-1$. It follows that $P$ must contain all $k^2$ of the $3$'s. But then, since $$D_{n,k} = (n-1)n^{k-1}\ldots 34^{k-1}W$$ (by construction), we deduce that $P$ must contain all $k^2$ of the $4$'s that occur in $D_{n,k}$, and hence all $k^2$ of the $5$'s, and so on. That is, $P$ must contain all $k^2$ occurrences of the letter $x$ for $x = 3, \ldots, n-1$; whence, since $D_{n,k}$ begins with the letter $n-1$, we have $P = \Omega_{n,1}\Omega_{n,2} \cdots \Omega_{n,k} = D_{n,k}$, a contradiction. 

Thus, $P$ overlaps exactly two adjacent blocks in $D_{n,k}$, in which case $P$ cannot contain the letter $2$; otherwise $P$ would contain all $k$ of the $2$'s, and hence would overlap all of the blocks in $D_{n,k}$, which is impossible (by the preceding arguments). Hence, $P$ lies strictly between two successive occurrences of the letter $2$ in $D_{n,k}$. But then $P$ cannot contain the letter $n$ as there are exactly $k-2$ occurrences of the letter $n$ between successive $2$'s in $D_{n,k}$ (by Fact~3). Therefore, since the blocks $\Omega_{n,i}$ with $2 \leq i \leq k-1$ end with the letter $n$, it follows that $P$ overlaps the blocks $\Omega_{n,1}$ and $\Omega_{n,2}$. Now, by construction, $\Omega_{n,1}$ ends with $1^{k-1} = 11\ldots 1$ ($k-1$ times), and hence $P$ contains $k$ of the $2(k-1) = 2k-2$ occurrences of the letter~$1$ in $\Omega_{n,1}\Omega_{n,2}$. But then $P$ must contain the letter $2$ because $\Omega_{n,1}$ contains exactly $(k-1)$ occurrences of the letter $1$ (as a suffix) and all $(k-1)$ of the $1$'s in $\Omega_{n,2}$ occur after the letter $2$ (by Fact~2); a contradiction.

We have now shown that $D_{n,k}$ is abelian $k$-power-free. It remains to show that $D_{n,k}x$ ends with an abelian $k$-th power for each letter $x = 1, 2, \ldots, n$. \\

\noindent{\bf Cruciality:}  By construction, $D_{n,k}n$ is clearly an abelian $k$-th power. It is also easy to see that $D_{n,k}1$ ends with the (abelian) $k$-th power $\Delta_{1}1 := 11\ldots 1$ ($k$ times).  Furthermore, for all $m = n, n-1,\ldots, 4, 3$, we deduce from the construction that
\begin{eqnarray*}
\Omega_{m,1} &=& (m-1)m^{k-1}\Omega_{m-1,1}, \\
\Omega_{m,2} &=& (m-1)m^{k-2}\Omega_{m-1,2}m, \\
%\Omega_{m,3} &=& (m-1)m^{k-3}\Omega_{m-1,3}m^2, \\
&\vdots& \\
\Omega_{m,k-2} &=& (m-1)m^2\Omega_{m-1,k-2}m^{k-3}, \\
\Omega_{m,k-1} &=& m(m-1)\Omega_{m-1,k-1}m^{k-2}, \\
\Omega_{m,k} &=& (m-1)\Omega_{m-1,k}[1^{k-1}]^{-1}(m-1)m^{k-2}1^{k-1},
\end{eqnarray*}
where $\Omega_{m-1,k}[1^{k-1}]^{-1}$ indicates the deletion of the suffix $1^{k-1}$ of $\Omega_{m-1,k}$.

Consequently, for $x = n-1, n-2, \ldots, 3, 2$, the word $D_{n,k}x$ ends with the abelian $k$-th power $\Delta_x$ given by
\[
\Delta_{x+1} = x(x+1)^{k-1}\Delta_{x} \quad \mbox{where $\Delta_n = D_{n,k}$}.
\]
\end{proof}

Observe that $|D_{n,2}| = 4n-7$ and $|D_{n,3}| = 9n-13$. Hence, since $D_{n,k}$ is a crucial abelian $k$-power-free word (by the proof of Theorem~\ref{opt-upper-k}), it follows from \cite[Theorem~5]{EK} and Corollary~\ref{min-length} that the words $D_{n,2}$ and $D_{n,3}$ are minimal crucial words over $\A_n$ avoiding abelian squares and abelian cubes, respectively. That is, for $k = 2,3$, the word $D_{n,k}$ gives the length of a minimal crucial word over $\A_n$ avoiding abelian $k$-th powers. In the case of $k\geq 4$, we make the following conjecture.

\begin{conjecture}\label{conj} For $k\geq 4$ and sufficiently large $n$, the length of a minimal crucial word over $\A_n$ avoiding abelian $k$-th powers
is given by $k^2(n-1)-k-1$. \end{conjecture}

\subsection{A lower bound for $\ell_k(n)$}

A trivial lower bound for $\ell_k(n)$ is $nk-1$ as all letters except
$n$ must occur at least $k$ times, whereas $n$ must occur at least
$k-1$ times. We give here the following slight improvement of the
trivial lower bound, which must be rather imprecise though.

\begin{theorem} \label{rough-l-bound} For $n\geq 5$ and $k\geq 4$, we have $\ell_k(n)\geq k(3n-4)-1$.\end{theorem}

\begin{proof} Notice that in proving Lemmas~\ref{lem1}--\ref{lem4} we do not use the fact that one deals with abelian cube-free words,
which we use to obtain a lower bound for $\ell_k(n)$. Indeed,
assuming that $X$ is a crucial word over the $n$-letter alphabet
$\A_n$ with respect to abelian $k$-th powers ($k\geq 4$), we see that
adjoining any letter from $\A_n$ to the right of $X$ must create a
cube as a factor from the right. In particular, adjoining $n$ from
the right side leads to creating a cube of length at least $9n-13$
(by Lemmas~\ref{lem1}--\ref{lem4}). This cube will be
$\Omega_{n,k-2}\Omega_{n,k-1}\Omega'_{n,k}$ in $X$ and thus
$|\Omega_{n,i}|$, for $1\leq i\leq k-1$, will have length at least
$3n-4$, whereas $|\Omega'_{n,k}|$ has length at least $3n-5$,  which yields the result.
\end{proof}

\section{Further research}

\begin{enumerate}

\item Prove or disprove Conjecture~\ref{conj}. Notice that the general construction uses a
greedy algorithm for going from $k-1$ to $k$, which does not work
for going from $n-1$ to $n$ for a fixed $k$. However, we believe
that the conjecture is true.

\item A word $W$ over $\A_n$ is {\em maximal} with respect to a given set of prohibitions if $W$ avoids the prohibitions, but $xW$ and $Wx$ do not avoid the prohibitions for any letter $x\in\A_n$. A maximal word of the shortest length is called a {\em minimal maximal word}. Clearly,
the length of a minimal crucial word with respect to a given set of
prohibitions is at most the length of a minimal maximal word. Thus,
obtaining the length of a minimal crucial word we get a lower bound
for the length of a minimal maximal word.

Can we use our approach to tackle the problem of finding minimal
maximal words? In particular, Korn~\cite{K} proved that the length
$\ell(n)$ of a shortest maximal abelian square-free word over $\A_n$ satisfies
$4n-7\leq\ell(n)\leq 6n-10$ for $n\geq 6$, while Bullock~\cite{B}
refined Korn's methods to show that $6n-29\leq\ell(n)\leq 6n-12$ for
$n\geq 8$. Can our approach improve this result (probably too much to
ask when taking into account how small the gap is), or can it provide an alternative solution to Bullock's result?
\end{enumerate}

\bigskip

\end{document}